\documentclass[11pt]{amsart}
\usepackage{a4wide}
\usepackage{amssymb}
\usepackage{mathrsfs}
\usepackage{enumerate}
\theoremstyle{plain}
\newtheorem{theorem}{Theorem}[section]
\newtheorem*{theorem*}{Theorem}
\newtheorem{lemma}[theorem]{Lemma}
\newtheorem{proposition}[theorem]{Proposition}

\newtheorem*{remark*}{Remark}
\newtheorem{itheorem}{Theorem}

\newtheorem{icorollary}[itheorem]{Corollary}
\theoremstyle{definition}

\numberwithin{equation}{section}

\newcommand{\NN}{\mathbf{N}}
\newcommand{\RR}{\mathbf{R}}
\newcommand{\ZZ}{\mathbf{Z}}
\newcommand{\se}{\subseteq}

\newcommand{\fhi}{\varphi}
\newcommand{\teta}{\vartheta}
\newcommand{\ww}[1]{|#1|_{\mathrm{w}}}
\newcommand{\pf}{\mathscr{P}_\mathrm{f}}

\newcommand{\full}[1]{[[#1]]}
\newcommand{\sbar}{\mbox{\footnotesize $|$}}

\begin{document}

\title[Cantor systems, piecewise translations and simple amenable groups]{Cantor systems, piecewise translations\\ and simple amenable groups}
\author[K. Juschenko N. Monod]{Kate Juschenko and Nicolas Monod}
\address{EPFL, 1015 Lausanne, Switzerland}
\thanks{Work supported in part by the European Research Council and the Swiss National Science Foundation.}
%
\begin{abstract}
We provide the first examples of finitely generated simple groups that are amenable (and infinite). This follows from a general existence result on invariant states for piecewise-translations of the integers. The states are obtained by constructing a suitable family of densities on the classical Bernoulli space.
\end{abstract}
\maketitle

\section{Introduction}
A \emph{Cantor system} $(T,C)$ is a homeomorphism $T$ of the Cantor space $C$; it is called \emph{minimal} if $T$ admits no proper invariant closed subset. The \emph{topological full group} $\full{T}$ of a Cantor system is the group of all homeomorphisms of $C$ which are given piecewise by powers of $T$, each piece being open in $C$. This countable group is a complete invariant of flip-conjugacy for $(T,C)$ by a result of Giordano--Putnam--Skau~\cite[Cor.~4.4]{Giordano-Putnam-Skau99}.

\medskip
It turns out that this construction yields very interesting groups $\full{T}$. Indeed, Matui proved that the commutator subgroup of $\full{T}$ is simple for any minimal Cantor system, see Theorem~4.9 in~\cite{Matui06} and the remark preceding it. Moreover, he showed that this simple group is finitely generated if and only if $(T,C)$ is (conjugated to) a minimal subshift. This yields a new uncountable family of non-isomorphic finitely generated simple groups since subshifts can be distinguished by their entropy; see~\cite[p.~246]{Matui06}.

\medskip
Until now, no example of finitely generated simple group that is  \emph{amenable} (and infinite) was known. Grigorchuk--Medynets~\cite{Grigorchuk-Medynets_v3} have proved that the topological full group $\full{T}$ of a minimal Cantor system $(T,C)$ is locally approximable by finite groups in the Chabauty topology. They conjectured that $\full{T}$ is amenable; our first result confirms this conjecture.

\begin{itheorem}\label{thm:main}
The topological full group of any minimal Cantor system is amenable.
\end{itheorem}

\noindent
Surprisingly, this statement fails as soon as one allows two commuting homeomorphisms~\cite{Elek-Monod}.

\bigskip
Combining Theorem~\ref{thm:main} with the above-mentioned results from~\cite{Giordano-Putnam-Skau99,Matui06} we deduce:

\begin{icorollary}
There exist finitely generated simple groups that are infinite amenable. In fact, there are $2^{\aleph_0}$ non-isomorphic such groups.\qed
\end{icorollary}

In order to prove Theorem~\ref{thm:main}, we reformulate the problem in terms of the group $W(\ZZ)$ of \emph{piecewise-translations} of the integers. More precisely, we denote by $W(\ZZ)$ the group of all those bijections $g$ of $\ZZ$ for which the quantity
$$\ww{g} :=\sup\big\{| g (j)-j|:j\in\ZZ \big\}$$
is finite. The topological full group of any minimal Cantor system $(T,C)$ can be embedded into $W(\ZZ)$ by identifying a $T$-orbit with $\ZZ$. However, $W(\ZZ)$ also contains many other groups, including non-abelian free groups. This fact can be traced back to Schreier's 1927 proof of the residual finiteness of free groups, see~\S\,2 in~\cite{Schreier27} (or~\cite{vanDouwen} for a more modern viewpoint).

\medskip
We shall introduce a model for random finite subsets of $\ZZ$ which has the following two properties: (i)~the model is almost-invariant under shifts by piecewise-translations; (ii)~a random finite set contains~$0$ with overwhelming probability. Theorem~\ref{thm:main} is proved using a general result about $W(\ZZ)$ which has the following equivalent reformulation.

\begin{itheorem}\label{thm:wobbling}
The $W(\ZZ)$-action on the collection of finite sets of integers admits an invariant mean which gives full weight to the collection of sets containing~$0$.
\end{itheorem}

Notice that for any given finite set $E\se\ZZ$, a mean as in Theorem~\ref{thm:wobbling} will give full weight to the collection of sets containing~$E$.

\subsection*{Acknowledgements}
Part of this work was done when the authors enjoyed the hospitality of the Mittag-Leffler institute. We are indebted to G.~Elek for inspiring conversations. We thank Y.~de Cornulier and J.~Peterson for useful comments on an earlier draft. The possibility that topological full groups of minimal Cantor systems could be amenable was suggested by R.~Grigorchuk and K.~Medynets.

\section{Semi-densities on the Bernoulli shift}\label{sec:ai}
The technical core of our construction is a family of $L^2$-functions $f_n$ on the classical Bernoulli space $\{0,1\}^{\ZZ }$. The relevance of these functions will be explained in Section~\ref{sec:actions}.

\bigskip
For any $n\in \NN$, we define
$$f_n\colon \{0,1\}^{\ZZ } \longrightarrow (0,1], \kern5mm f_n(x)=\exp\Bigg(-n\sum\limits_{j\in\ZZ }x_je^{-|j|/n}\Bigg),$$
where $x=\{x_j\}_{j\in\ZZ }\in \{0,1\}^{\ZZ }$. We consider $f_n$ as an element of the Hilbert space $L^2(\{0,1\}^{\ZZ })$, where $\{0,1\}^{\ZZ }$ is endowed with the symmetric Bernoulli measure. The interest of the family $f_n$ is that it satisfies the following two properties, each of which would be elementary to obtain separately.

\begin{theorem}\label{thm:estimate}
For any $g\in W(\ZZ)$ we have  $\langle  g (f_n),f_n\rangle / \|f_n\|^2\rightarrow 1$ as $n\rightarrow \infty$. Moreover, $\|f_n\sbar_{x_0=0} \|/\|f_n\| \to 1$.
\end{theorem}

\noindent
The notation $f_n\sbar_{x_0=0}$ represents the function $f_n$ multiplied by the characteristic function of the cylinder set describing the elementary event $x_0=0$.

\bigskip
In preparation for the proof, we write
$$a_{n,j} = \exp(-ne^{-|j|/n})\kern5mm \text{for} \kern3mm j\in\ZZ.$$
We shall often use implicitly the estimates
$$0 <  a_{n, j} \leq 1 \kern5mm\text{and}\kern5mm 0 <  \frac{a_{n, j}^2}{1+a_{n, j}^2} \leq a_{n, j}^2 \leq a_{n, j}.$$
Since $f_n$ is a product of the independent random variables $\exp\big(-n x_je^{-|j|/n}\big)$, we have
$$\|f_n\|^2 = \prod_{j\in\ZZ } \left( \frac12 + \frac12 a_{n,j}^2 \right).$$
A straightforward estimate shows that this product converges unconditionally (in the sense that the series of $\log\big(\frac12+ \frac12 a_{n,j}^2 \big)$ converges absolutely). We can regroup factors and compute the ratio
$$\frac{\|f_n\sbar_{x_0=0} \|^2}{\|f_n\|^2} =\frac1{1 + a_{n, 0}^2}$$
which thus converges to~$1$ as desired for the second statement of Theorem~\ref{thm:estimate}.

\medskip
The proof of the first statement will be divided into two propositions. Define the function $F_n \colon W(\ZZ)\to\RR$ by
$$F_n(g) = \sum_{j\in\ZZ} \frac{a_{n, j}^2}{1+a_{n, j}^2} \ e^{-|j|/n}\ \Big( |g(j)| - |j|\Big).$$
We begin with a conditional convergence:

\begin{proposition}\label{prop:mainestimate}
For any $g \in W(\ZZ)$ we have  $\langle  g (f_n),f_n\rangle / \|f_n\|^2\rightarrow 1$ as $n\rightarrow \infty$ provided $F_n(g)\to 0$.
\end{proposition}

The condition $F_n(g)\to 0$ is about a \emph{signed} series for which no absolute convergence to zero holds; it will be addressed by the following statement:

\begin{proposition}\label{prop:vanishes}
We have $\lim_{n\to\infty} F_n(g)=0$ for every $g \in W(\ZZ)$.
\end{proposition}

We now undertake the proof of Proposition~\ref{prop:mainestimate}. Using again the product form of $f_n$, one obtains
\begin{align*} 
\frac{\langle  g (f_n),f_n\rangle}{\|f_n\|^2}&=\prod\limits_{j\in\ZZ } \frac{1+a_{n,j}a_{n, g (j)}}{1+a_{n,j}^2}.
\end{align*}
Thus $\langle  g (f_n),f_n\rangle / \|f_n\|^2\rightarrow 1$ if and only if
\begin{equation}\label{eq:sum_log}
\lim_{n\to\infty} \sum_{j\in\ZZ } \log \frac{1+a_{n,j}a_{n, g(j)}}{1+a_{n,j}^2} =0.
\end{equation}
Next, we point out the elementary fact that there is an absolute constant $C>0$ (namely $C=4\log2 -2$) such that
\begin{equation}\label{eq:log}
z - C z^2 \leq \log (1+z) \leq z \kern5mm\forall\,z\geq -\frac12.
\end{equation}
We can apply this inequality to each summand of the series in~(\ref{eq:sum_log}) by writing
$$\frac{1+a_{n,j}a_{n, g (j)}}{1+a_{n,j}^2}= 1+ \frac{a_{n,j}^2}{1+a_{n,j}^2}\left( \frac{a_{n, g (j)}}{a_{n,j}}-1 \right)$$
because $0< a_{n,j}\leq 1$ for all $n$ and $j$ implies
$$ \frac{a_{n,j}^2}{1+a_{n,j}^2}\left( \frac{a_{n, g (j)}}{a_{n,j}}-1 \right)\geq -\frac{a_{n,j}^2}{1+a_{n,j}^2}\geq -\frac{1}{2}.$$
Therefore, summing up the inequalities given by~(\ref{eq:log}), we conclude that Proposition~\ref{prop:mainestimate} will follow once we prove the following two facts:
\begin{gather}\label{sum1}
\sum_{j\in\ZZ } \left(\frac{a_{n,j}^2}{1+a_{n,j}^2}\right)^2\left( \frac{a_{n, g (j)}}{a_{n,j}}-1 \right)^2 \rightarrow 0 \kern5mm\forall\,g\in W(\ZZ),
\end{gather}
 \begin{gather}\label{sum2}
 \sum_{j\in\ZZ } \frac{a_{n,j}^2}{1+a_{n,j}^2}\left( \frac{a_{n, g (j)}}{a_{n,j}}-1 \right)\rightarrow 0   \kern5mm\forall\,g\in W(\ZZ) \kern3mm \text{provided} \kern3mmF_n(g)\to 0.
\end{gather} 
Here is our first lemma.

\begin{lemma}\label{lem:sum}
For all $n$ we have
$$\sum_{j\in\ZZ} a_{n,j} e^{-|j|/n} \leq 3 \kern5mm\text{and}\kern5mm \sum_{j\in\ZZ} a_{n,j}^2 e^{-2|j|/n} \leq \frac1n.$$
\end{lemma}

It is based on the following elementary comparison argument.

\begin{lemma}\label{lem:integral}
Let $t_0\geq 0$ and let $\fhi\colon \RR_{\geq 0} \to \RR_{\geq 0}$ be a function which is increasing on $[0, t_0]$ and decreasing on $[t_0, \infty)$. Then
$$\sum_{j\geq 0} \fhi(j) \leq  \fhi(t_0) + \int_0^\infty \fhi(t)\,dt. \eqno{\square}$$
\end{lemma}

\begin{proof}[Proof of Lemma~\ref{lem:sum}]
For the first series, we consider the function $\fhi$ defined by $\fhi(t) = \exp(-n e^{-t/n}) e^{-t/n}$. One verifies that it satisfies the condition of Lemma~\ref{lem:integral} for $t_0=n\log n$. Therefore we can estimate
$$\sum_{j\in\ZZ} a_{n,j} e^{-|j|/n}  < 2 \sum_{j\geq 0} \fhi(j) \leq 2  e^{-1}/n + 2\int_0^\infty \exp(-n e^{-t/n})  e^{-t/n}\,dt.$$
The change of variable $s=e^{-t/n}$ shows that the integral is $\int_0^1 n e^{-ns} \,ds = 1-e^{-n}$ and thus in particular the series is bounded by $2(e^{-1}+1) <3$. For the second series, consider $\fhi(t) =  \exp(-2n e^{-t/n}) e^{-2t/n}$, again with $t_0=n\log n$. Lemma~\ref{lem:integral} yields
$$\sum_{j\in\ZZ} a_{n,j}^2 e^{-2|j|/n}  < 2 \sum_{j\geq 0} \fhi(j) \leq 2 (n e)^{-2}+ 2 \int_0^\infty \exp(-2n e^{-t/n})  e^{-2t/n}\,dt.$$
The change of variable $s=e^{-t/n}$ shows that the integral is
$$\int_0^1 n e^{-2ns} s\,ds = \frac{1-(1+2n)e^{-2n}}{4n} < \frac1{4n}$$
and thus in particular the series is bounded by $2 (n e)^{-2} + 1/(2n)< 1/n$.
\end{proof}

\begin{lemma}\label{firstestimate}
For any $g\in W(\ZZ)$ there are constants $C_g $, $C_g '$ and $C_g  ''$ which depend only on $\ww{g}$ such that for all $n$ and $j$ we have:
\begin{equation}
\frac{a_{n, g (j)}}{a_{n,j}} = \exp\Big(e^{-\frac{|j|}{n}}(| g (j)|-|j|+\eta( g ,j,n))\Big), \kern2mm\text{where}\kern2mm |\eta( g ,j,n)|\leq C_g /n.\label{C1}
\end{equation}
\begin{multline}
\frac{a_{n, g (j)}}{a_{n,j}}-1 = e^{-\frac{|j|}{n}}(| g (j)|-|j|)+ \eta({ g },n,j)e^{-\frac{|j|}{n}}+\teta({ g },n,j),\\
\text{where\kern2mm} |\teta({ g },n,j)| \leq C_{ g }' e^{-2\frac{|j|}{n}}.\label{C2}
\end{multline}
\begin{equation}
\left|\frac{a_{n, g (j)}}{a_{n,j}}-1 \right| \leq C_{ g }'' e^{-\frac{|j|}{n}}.\label{C3}
\end{equation}
\end{lemma}

\begin{proof}
Note that the conclusion~(\ref{C3}) is an easy consequence of~(\ref{C1}) and~(\ref{C2}). From the definition of $a_{n,j}$ we have 
$$\frac{a_{n, g (j)}}{a_{n,j}}=\exp\left(e^{-\frac{|j|}{n}}n\left(1-e^{\frac{|j|-| g (j)|}{n}}\right)\right).$$
Then using the Taylor expansion we have
$$n\left(1-e^{\frac{|j|-| g (j)|}{n}}\right)=| g (j)|-|j|  +\eta( g ,j,n),$$
wherein
$$\eta( g ,j,n):=-\sum_{k\geq 2} \frac{(|j|-| g (j)|)^k}{k!n^{k-1}}.$$
Now
$$|\eta( g ,j,n)|\leq \frac{1}{n} \sum_{k\geq 2} \frac{\ww{g}^k}{k!}\leq \frac{e^{\ww{g}}}{n}$$
which proves~(\ref{C1}). Continuing to expand~(\ref{C1}), we have
\begin{align*}
\frac{a_{n, g (j)}}{a_{n,j}}-1&=\exp\Big(e^{-\frac{|j|}{n}}(| g (j)|-|j|+\eta( g ,j,n))\Big)-1=\\
&= e^{-\frac{|j|}{n}}(| g (j)|-|j|)+ e^{-\frac{|j|}{n}}\eta( g ,j,n)+\teta( g ,j,n)
\end{align*}
wherein
$$\teta( g ,j,n):=\sum_{k\geq 2}\frac{1}{k!}e^{-\frac{k|j|}{n}}\big(| g (j)|-|j|+\eta( g ,j,n)\big)^k.$$
Thus we have
\begin{align*}
|\teta( g ,j,n)|&\leq  e^{-\frac{2|j|}{n}}  \sum_{k\geq 2}\frac{1}{k!}\Big| | g (j)|-|j|+\eta( g ,j,n)\Big|^k \leq e^{-\frac{2|j|}{n}} \exp\Big({\ww{g}+\frac{C_{ g }}{n}}\Big) \leq e^{-\frac{2|j|}{n}}C_{ g }',
\end{align*}
as required for~(\ref{C2}).
\end{proof}

\begin{proof}[End of the proof of Proposition~\ref{prop:mainestimate}]
Recall that we have reduced the proof to showing~(\ref{sum1}) and~(\ref{sum2}). By Lemma~\ref{firstestimate}(\ref{C3}) and Lemma~\ref{lem:sum} we have
$$\sum_{j\in\ZZ } \left(\frac{a_{n,j}^2}{1+a_{n,j}^2}\right)^2\left( \frac{a_{n, g (j)}}{a_{n,j}}-1 \right)^2 \leq C_{ g }''^2 \sum_{j\in\ZZ } a_{n,j}^4 e^{-2\frac{|j|}{n}}\leq  C_{ g }''^2 \sum_{j\in\ZZ } a_{n,j}^2 e^{-2\frac{|j|}{n}}\leq C_{ g }''^2 /n,$$
which implies the convergence~(\ref{sum1}). For~(\ref{sum2}), keep the notations of Lemma~\ref{firstestimate}. By point~(\ref{C2}) of that lemma, we have
\begin{align*} \sum_{j\in\ZZ } \frac{a_{n,j}^2}{1+a_{n,j}^2}\left( \frac{a_{n, g (j)}}{a_{n,j}}-1 \right)&= \sum_{j\in\ZZ }\frac{a_{n,j}^2}{1+a_{n,j}^2}e^{-\frac{|j|}{n}} \Big(| g (j)|-|j|\Big) \\
&+ \sum_{j\in\ZZ }\frac{a_{n,j}^2}{1+a_{n,j}^2}e^{-\frac{|j|}{n}} \eta( g ,j,n) \\
&+ \sum_{j\in\ZZ }\frac{a_{n,j}^2}{1+a_{n,j}^2}\teta( g ,j,n)
\end{align*}
and we recall that the first of the three terms is $F_n(g)$, which is assumed to go to zero. For the second term, since $|\eta( g ,j,n)|\leq C_g /n$, Lemma~\ref{lem:sum} gives
$$\left| \sum_{j\in\ZZ }\frac{a_{n,j}^2}{1+a_{n,j}^2}e^{-\frac{|j|}{n}} \eta( g ,j,n)\right| \leq  \frac{C_g }{n} \sum_{j\in\ZZ } a_{n,j} e^{-\frac{|j|}{n}} \leq \frac{3 C_g}{n}.$$
For the last term, since $|\teta( g ,j,n)|\leq C_{ g }' e^{-2\frac{|j|}{n}}$, Lemma~\ref{lem:sum} implies
$$\left| \sum_{j\in\ZZ }\frac{a_{n,j}^2}{1+a_{n,j}^2}\teta( g ,j,n) \right| \leq C_{ g }' \sum_{j\in\ZZ} a_{n,j}^2 e^{-2\frac{|j|}{n}} \leq  \frac{C_{ g }'}{n}.$$
This completes the proof of~(\ref{sum2}) and therefore of the proposition.
\end{proof}

In order to apply Proposition~\ref{prop:mainestimate}, we need to control $F_n$ as stated in Proposition~\ref{prop:vanishes}.  Let thus $g\in W(\ZZ)$ be given; writing
$$b_0 =|g(0)|, \kern5mm \text{and} \kern3mm b_j=|g(j)|+|g(-j)|-(|j|+|-j|) \kern5mm\text{for}\kern3mm j>0,$$
we have
$$F_n(g)=\sum_{j=0}^{\infty}\frac{a_{n, j}^2}{1+a_{n, j}^2}e^{-j/n}b_j$$
since $a_{n, j}=a_{n, -j}$. Define functions $B$ and $\psi$ on $\RR_{\geq 0}$ by
$$B(t)=\sum_{0\leq j\leq t}b_j, \kern5mm \psi(t)=\frac{\exp(-2ne^{-t/n})}{1+\exp(-2ne^{-t/n})}e^{-t/n}.$$
Then the Abel summation formula gives
\begin{equation}\label{ab}
\sum_{j=0}^N \psi(j)b_j=\psi(N)B(N)-\int_0^N B(t)\, d\psi(t). \kern10mm(\forall\,N\in\NN)
\end{equation}
\begin{lemma}\label{lem:B}
We have $-2\ww{g}^2\leq B(u)\leq 4\ww{g}^2$ for all $u>\ww{g}$.
\end{lemma}

\begin{proof}
For simplicity, write $c:=\ww{g}$ and $J_u:=\{j:|j|\leq u\}$. Thus $B(u)=\sum_{j\in g(J_u)}|j|-\sum_{j\in J_u}|j|$. Since $J_{u-c}\se g(J_u)$, we have
\begin{equation}\label{eq:B}
B(u)=\sum_{j\in g(J_u)}|j|-\sum_{j\in J_u}|j|=\sum_{j\in g(J_u)\setminus J_{u-c}}|j|-\sum_{j\in J_u\setminus J_{u-c}}|j|.
\end{equation}
Now note first that since $J_{u-c}\se g(J_u)$, the number of elements
in the set $g(J_u)\setminus J_{u-c}$ is equal to the number of
elements in $J_u\setminus J_{u-c}$, which is $2c$. Also, for any $j\in
g(J_u)\setminus J_{u-c}$ we have $u-c < |j|\leq u+c$, and for any $j\in J_u\setminus J_{u-c}$, $u-c<|j|\leq u$. Hence~\eqref{eq:B} implies
$$-2c^2=2c(u-c)-2cu\leq B(u)\leq 2c(u+c)-2c(u-c)=4c^2.$$
\end{proof}

\begin{proof}[End of the proof of Proposition~\ref{prop:vanishes}.]
Since $B(N)$ is bounded by Lemma~\ref{lem:B} and since $\lim_{N\to\infty}\psi(N)$ vanishes, the equality~\eqref{ab} gives $F_n(g) = -\int_0^\infty B(t)\,d\psi(t)$. After computing explicitly the derivative $\psi'$, this rewrites as
$$F_n(g)= \frac{1}{n}\int_0^{\infty}B(t)\psi(t)dt -\int_0^{\infty}B(t)\frac{2\exp(-2ne^{-t/n}) e^{-2t/n}}{(1+\exp(-2ne^{-t/n}))^2} dt.$$
Using Lemma~\ref{lem:B} and $0<\psi(t)\leq \exp(-ne^{-t/n})e^{-t/n}$, the first integral is bounded by
$$\left|\frac{1}{n}\int_0^{\infty}B(t)\psi(t)dt\right|\leq \frac{1}{n}4\ww{g}^2 \int_0^{\infty}\exp(-ne^{-t/n})e^{-t/n}dt = \frac{1}{n}4\ww{g}^2 (1-e^{-n}),$$
which goes to zero. Similarly, the second integral is bounded by
$$\left|\int_0^{\infty}B(t)\frac{2\exp(-2ne^{-t/n})}{(1+\exp(-2ne^{-t/n}))^2}e^{-2t/n}dt\right|\leq 8\ww{g}^2\int_0^{\infty}\exp(-2ne^{-t/n})e^{-2t/n}dt < \frac{2\ww{g}^2}{n},$$
the last inequality having already been observed in the proof of Lemma~\ref{lem:sum}.
\end{proof}

Taken together, Proposition~\ref{prop:vanishes} and Proposition~\ref{prop:mainestimate} finish the proof of Theorem~\ref{thm:estimate} since we already observed  $\|f_n\sbar_{x_0=0} \|/\|f_n\| \to 1$.

\section{Actions on sets of finite subsets}\label{sec:actions}
Let $G$ be a group acting on a set $X$. The collection $\pf(X)$ of finite subsets of $X$ is an abelian $G$-group for the operation~$\triangle$ of symmetric difference. The resulting semi-direct product $\pf(X)\rtimes G$, which can be thought of as the ``lamplighter'' restricted wreath product associated to the $G$-action on $X$, has itself a natural ``affine'' action on $\pf(X)$, where the latter set can be considered as the coset space $(\pf(X)\rtimes G) /G$.

\medskip
It will be convenient to identify the Pontryagin dual of the (discrete) group $\pf(X)$ with the generalised Bernoulli $G$-shift $\{0,1\}^X$, the duality pairing being given for $E\in\pf(X)$ and $\omega=\{\omega_x\}_{x\in X}\in \{0,1\}^X$ by the character $\exp(i \pi \sum_{x\in E} \omega_x)\in\{\pm 1\}\se \mathbf{C}^*$. The normalised Haar measure corresponds to the symmetric Bernoulli measure on $\{0,1\}^X$.

\begin{lemma}\label{lem:Fourier}
Assume that $G$ acts transitively on $X$ and choose $x_0\in X$. The following assertions are equivalent.
\begin{enumerate}[(i)]
\item There is a net $\{f_n\}$ of $G$-almost invariant vectors in $L^2(\{0,1\}^X)$ such that the ratio $\|f_n\sbar_{\omega_{x_0}=0}\| / \|f_n\|$ converges to~$1$.\label{Fourier:ai}
\item The $\pf(X)\rtimes G$-action on $\pf(X)$ admits an invariant mean.\label{Fourier:mean}
\item The $G$-action on $\pf(X)$ admits an invariant mean giving weight~$1/2$ to the collection of sets containing~$x_0$.\label{Fourier:05}
\item The $G$-action on $\pf(X)$ admits an invariant mean giving full weight to the collection of sets containing~$x_0$.\label{Fourier:1}
\end{enumerate}
\end{lemma}

\noindent
Again, $f_n\sbar_{\omega_{x_0}=0}$ denotes the function $f_n$ multiplied by the characteristic function of the cylinder set describing the elementary event $\omega_{x_0}=0$. The net $\{f_n\}$ can of course be chosen to be a sequence when $G$ (and hence $X$) is countable.

\begin{proof}[Proof of Lemma~\ref{lem:Fourier}]
(\ref{Fourier:ai})$\Longrightarrow$(\ref{Fourier:mean}).
The Fourier transform $\widehat{f_n}$ provides $G$-almost invariant vectors in $\ell^2(\pf(X))$. Morover, $\|f_n\sbar_{\omega_{x_0}=0}\|$ is the norm of the image of $\widehat{f_n}$ projected to the subspace of vectors in $\ell^2(\pf(X))$ that are invariant under $\{x_0\}$ viewed as group element in $\pf(X)$. Thus $\widehat{f_n}$ is $\{x_0\}$-almost invariant. Since the $G$-action is transitive, it follows that $\widehat{f_n}$ is $\pf(X)$-almost invariant as $n\to\infty$.

\smallskip
(\ref{Fourier:mean})$\Longrightarrow$(\ref{Fourier:05}).
The condition on $x_0$ follows from the invariance under $\{x_0\}$.

\smallskip
(\ref{Fourier:05})$\Longrightarrow$(\ref{Fourier:1}).
It suffices to show that for each $k\in\NN$ there are $G$-almost-invariant probability measures on $\pf(X)$ such that the collection of sets containing~$x_0$ has probability at least $1-2^{-k}$. By~(\ref{Fourier:05}), we have $G$-almost-invariant probability measures such that the collection of sets containing~$x_0$ has probability~$1/2$. Indeed, the classical proof of the ``Reiter property'' produces almost invariant probability measures as convex combinations of a net approximating an invariant mean in the weak-* topology, and our restriction about $x_0$ is preserved under convex combinations. If we take the union of $k$ independently chosen such finite sets, we obtain a distribution as required.

\smallskip
(\ref{Fourier:1})$\Longrightarrow$(\ref{Fourier:ai}).
The assumption implies that there are $G$-almost-invariant probability measures $\mu$ on $\pf(X)$ such that the collection of sets containing~$x_0$ has probability~$1$, making the same observation about Reiter's property as in (\ref{Fourier:05})$\Rightarrow$(\ref{Fourier:1}). We can assume that each $\mu$ is supported on a collection of sets of fixed cardinal $n(\mu)\in\NN$. We define a function $f_\mu$ on $\{0,1\}^X$ as follows. Given $E\in\pf(X)$, consider the cylinder set $C_E\se \{0,1\}^X$ consisting of all $\omega$ such that $\omega_x=0$ for all $x\in E$. We set $f_\mu=2^{n(\mu)}\sum_{E\in \pf(X)}  \mu(\{E\}) 1_{C_E}$, where $1_{C_E}$ is the characteristic function of $C_E$. Then $f_\mu$ is supported on $\{\omega_{x_0}=0\}$, has $L^1$-norm one and satisfies $\|g f_\mu - f_\mu\|_1\leq \|g \mu - \mu\|_1$ for all $g\in G$. Therefore, the function $f_\mu^{1/2}$ is as required by~(\ref{Fourier:ai}) as $\mu$ becomes increasingly invariant since $\|g f_\mu^{1/2} - f_\mu^{1/2}\| \leq \|g f_\mu - f_\mu\|_1^{1/2}$.
\end{proof}

\begin{proof}[Proof of Theorem~\ref{thm:wobbling}]
The sequence $\{f_n\}$ constructed in Section~\ref{sec:ai} satisfies the criterion~(\ref{Fourier:ai}) of Lemma~\ref{lem:Fourier} in view of Theorem~\ref{thm:estimate}. Therefore, the criterion~(\ref{Fourier:1}) provides the desired conclusion.
\end{proof}

The following is well-known.

\begin{lemma}\label{lem:amen}
Let $H$ be a group acting on a set $Y$ with an invariant mean. If the stabiliser in $H$ of every $y\in Y$ is an amenable group, then $H$ is amenable.
\end{lemma}

\begin{proof}
The amenability of stabilisers implies that there is an $H$-map $Y\to \mathscr{M}(H)$ to the (convex compact) space $\mathscr{M}(H)$ of means on $H$ (by choosing for each $H$-orbit in $Y$ the orbital map associated to a mean fixed by the corresponding stabiliser). The push-forward of an invariant mean on $Y$ is an invariant mean on $\mathscr{M}(H)$. Its barycenter is an invariant mean on $H$. (An alternative argument giving explicit F\o lner sets can be found in the proof of Lemma~4.5 in~\cite{Glasner-Monod}.)
\end{proof}

The next proposition will leverage the fact that $\NN\triangle g(\NN)$ is finite for all $g\in W(\ZZ)$.

\begin{proposition}\label{prop:twist}
Let $G<W(\ZZ)$ be a subgroup such that the stabiliser in $G$ of $E\triangle \NN$ is amenable whenever $E\in\pf(\ZZ)$. Then $G$ is amenable.
\end{proposition}

\begin{proof}
As noted in the proof of Theorem~\ref{thm:wobbling}, the $W(\ZZ)$-action on $\ZZ$ satisfies the equivalent conditions of Lemma~\ref{lem:Fourier} thanks to Theorem~\ref{thm:estimate}. In particular, there is a $\pf(\ZZ)\rtimes G$-invariant mean on $\pf(\ZZ)$. Thus, in view of Lemma~\ref{lem:amen}, it suffices to find an embedding $\iota\colon G\to \pf(\ZZ)\rtimes G$ in such a way that the stabiliser in $\iota(G)$ of any finite set $E$ is the stabiliser in $G$ of $E\triangle \NN$. The map defined by $\iota(g) = \big(\NN\triangle g(\NN), g\big)$ has the required properties.
\end{proof}

\section{From Cantor systems to piecewise translations}
It is known that the stabiliser of a forward orbit in the topological full group of a minimal Cantor system is locally finite~\cite{Giordano-Putnam-Skau99}. The corresponding more general situation for the group $W(\ZZ)$ is described in the following two lemmas.

\medskip
A subgroup $G$ of $W(\ZZ)$ has the \emph{ubiquitous pattern property} if for every finite set $F\se G$ and every $n\in \NN$ there exists a constant $k=k(n,F)$ such that for every $j\in\ZZ$ there exists $t\in\ZZ$ such that $[t-n,t+n]\se [j-k,j+k]$ and such that for every $i\in [-n,n]$ and every $g\in F$ we have $g(i)=g(i+t)$.

Informally: the partial action of $F$ on $[-n,n]$ can be found, suitably translated, within any interval of length $2k+1$.

\begin{lemma}\label{lem:loc_fin}
Let $G<W(\ZZ)$ be a subgroup with the ubiquitous pattern property. Then the stabiliser of $E\triangle \NN$ in $G$ is locally finite for every $E\in \pf(\ZZ)$.
\end{lemma}

\begin{proof}
Let $E\in \pf(\ZZ)$ and $F$ be a finite set of elements of the stabiliser of $E\triangle \NN$ in $G$. In order to prove that the set $F$ generates a finite group it is sufficient to show that $\ZZ$ is a disjoint union of finite sets $B_i$ of uniformly bounded cardinality such that each of this sets is invariant under the action of $F$, since this will realize the group generated by $F$ as a subgroup of a power of a finite group. We will achieve this by taking the $B_i$ to be the ubiquitous translated copies of  the ``phase transition'' region of $E\triangle \NN$, suitably identifying the ``top part'' of $E\triangle \NN$ with the ``bottom part'' of the complement of the next translated copy.

\smallskip
Let $c=\max\{|e|: e\in E\}$ (with $c=0$ if $E=\varnothing$). Consider the interval $[-c-2m, c+2m]$, where $m=\max \{\ww{g}:g\in F\}$. Let $k=k(c+2m, F)$ be the constant from the definition of the ubiquitous pattern property. Denote $E_0=E\triangle\NN\cap [-c-2m, c+2m]$. Consider $\ZZ$ as disjoint union of consecutive intervals $I_i$ ($i\in\ZZ$) of length $2k+1$ such that $[-c-2m, c+2m]\se I_0$. Then, by the ubiquitous pattern property, for each interval $I_i$ there exists a set $E_i\se I_i$ (a translate of $E_0$) such that the action of $F$ on $E_i$ coincides with the action of $F$ on  $E_0$. Let
$$B_i=\Big(E_i\cup [\max(E_i)+1,\max(E_{i+1})]\Big) \setminus E_{i+1}.$$
It is easy to see that $\ZZ=\bigsqcup B_i$ and that each $B_i$ is $F$-invariant. Moreover, since $B_i\se I_i\cup I_{i+1}$, we have $|B_i|\leq 4k+2$ for all $i$.
\end{proof}

Let $T$ be a homeomorphism of a Cantor space $C$ and choose a point $p\in C$. If $T$ has no finite orbits, then we can define a map
$$\pi_p\colon \full{T} \longrightarrow W(\ZZ)$$
by the requirement
$$g(T^j p) = T^{\pi_p(g)(j)} p,\kern5mm (g\in\full{T}, j\in \ZZ).$$
The map $\pi_p$ is a group homomorphism and is injective if the orbit of $p$ is dense.

\begin{lemma}\label{lem:ubi}
If $T$ is minimal, then the image $\pi_p(\full{T})$ of the injective homomorphism $\pi_p$ has the ubiquitous pattern property.
\end{lemma}

\begin{proof}
For every $g\in \full{T}$ the sets $C_{g,i}= \{q : g(q) = T^i q\}$ define a clopen partition $C=\bigsqcup_{i\in\ZZ} C_{g,i}$ with all but finitely many $C_{g,i}$ empty. Suppose that the property fails. Then there is a finite set $F\se \full{T}$, an integer $n\in\NN$ and a sequence $\{j_k\}_{k\in\NN}$ in $\ZZ$ such that the none of the intervals $[j_k-k, j_k+k]$ contains any translated copy of the partial action of $\pi_p(F)$ on $[-n,n]$. Rephrased in $C$, this means the following. For every $t$ with $[t-n, t+n]\se[j_k-k, j_k+k]$, there is $g\in F$ such that the partition of $[t-n, t+n]$ induced by intersecting the $C_{g,i}$ with $\big\{ T^r p : r\in [t-n, t+n]\big\}$ is different from the partition that they induce on $[-n, n]$.

Consider now the set $M_k$ of all points $q\in C$ such that for every $t$ with $[t-n, t+n]\se[-k, k]$ there is $g\in F$ such that the partition of $[t-n, t+n]$ induced by intersecting the $C_{g,i}$ with $\big\{ T^r q : r\in [t-n, t+n]\big\}$ is different from the partition induced on $[-n, n]$. The set $M_k$ is non-empty because, in view of the previous observation, it contains $q=T^{j_k} p$. On the other hand, the successive $M_k$ form a decreasing sequence of closed subsets of $C$. Therefore, the intersection of all $M_k$ is a non-empty closed set. It is invariant by construction, but does not contain $p$ since $p\notin M_k$ as soon as $k\geq n$. This contradicts the minimality.
\end{proof}

\begin{proof}[Proof of Theorem~\ref{thm:main}]
By Lemma~\ref{lem:ubi}, the (injective image of the) topological full group $\full{T}$ has the ubiquitous pattern property. Therefore, Lemma~\ref{lem:loc_fin} shows that the stabiliser of $E\triangle \NN$ in $G$ is amenable for every $E\in \pf(\ZZ)$. Now Proposition~\ref{prop:twist} shows that $\full{T}$ is amenable.
\end{proof}

\bibliographystyle{../../BIB/amsalpha}
\bibliography{../../BIB/ma_bib}

\def\cprime{$'$}
\providecommand{\bysame}{\leavevmode\hbox to3em{\hrulefill}\thinspace}
\begin{thebibliography}{GPS99}

\bibitem[EM]{Elek-Monod}
G{\'a}bor Elek and Nicolas Monod, \emph{On the topological full group of a
  minimal cantor $\mathbf{Z}^2$-system}, preprint, to appear in Proc. AMS.

\bibitem[GM]{Grigorchuk-Medynets_v3}
Rostislav~I. Grigorchuk and Konstantin Medynets, \emph{Topological full groups
  are locally embeddable into finite groups}, Preprint, {\tt
  http://arxiv.org/abs/math/1105.0719v3}.

\bibitem[GM07]{Glasner-Monod}
Yair Glasner and Nicolas Monod, \emph{Amenable actions, free products and a
  fixed point property}, Bull. Lond. Math. Soc. \textbf{39} (2007), no.~1,
  138--150.

\bibitem[GPS99]{Giordano-Putnam-Skau99}
Thierry Giordano, Ian~F. Putnam, and Christian~F. Skau, \emph{Full groups of
  {C}antor minimal systems}, Israel J. Math. \textbf{111} (1999), 285--320.

\bibitem[Mat06]{Matui06}
Hiroki Matui, \emph{Some remarks on topological full groups of {C}antor minimal
  systems}, Internat. J. Math. \textbf{17} (2006), no.~2, 231--251.

\bibitem[Sch27]{Schreier27}
Otto Schreier, \emph{Die {U}ntergruppen der freien {G}ruppen}, Abhandlungen
  Math. Hamburg \textbf{5} (1927), 161--183.

\bibitem[vD90]{vanDouwen}
Eric~K. van Douwen, \emph{Measures invariant under actions of {$F\sb 2$}},
  Topology Appl. \textbf{34} (1990), no.~1, 53--68.

\end{thebibliography}

\end{document}